\documentclass[twoside, 10pt]{article}
\usepackage{amsmath}
\usepackage{latexsym}
\usepackage{amssymb}
\usepackage{graphicx}
\usepackage{amsfonts}
\usepackage{euscript}
\usepackage{epsfig}
\usepackage{color}
\usepackage{indentfirst}
\def\dj{d\kern-0.4em\char"16\kern-0.1em}
\def\Dj{\mbox{\raise0.3ex\hbox{-}\kern-0.4em D}}

\begin{document}

\markboth{Louis H. Kauffman, Ljubica S. Velimirovi\' c, Marija S. Najdanovi\' c and Svetozar R. Ran\v ci\' c }
{Infinitsimal bending of knots and energy change}


\title{Infinitesimal bending of knots and energy change}

\author{Louis H. Kauffman\\
\small{Department of Mathematics, Statistics and Computer Science,} \\
\small{851 South Morgan Street, University of Illinois at Chicago,} \\ 
\small{Chicago, Illinois 60607-7045  and} \\
\small{Department of Mechanics and Mathematics,} \\
\small{Novosibirsk State University, Novosibirsk, Russia;} \\
\small{ e-mail: kauffman@uic.edu} \\
Ljubica S, Velimirovi\' c \\
\small{Faculty of Science and Mathematics, University of Ni\v s, Ni\v s, Serbia;} \\
\small{e-mail: vljubica@pmf.ni.ac.rs} \\
Marija S. Najdanovi\' c \\
\small{Faculty of Natural Sciences and Mathematics,} \\
\small{University of Pri\v stina, Kosovska Mitrovica, Serbia;} \\
\small{e-mail: marijamath@yahoo.com} \\
Svetozar R. Ran\v ci\' c\\
\small{Faculty of Science and Mathematics, University of Ni\v s, Ni\v s, Serbia;}\\
\small{e-mail: rancicsv@yahoo.com}\\ 
}

\maketitle

\begin{abstract}
We discuss infinitesimal bending of curves and knots in ${\cal
R}^3$. A brief overview of the results on the infinitesimal bending
of curves is outlined. Change of the Willmore energy, as well as of
the M\" obius energy under infinitesimal bending of knots is
considered. Our visualization tool devoted to visual representation
of infinitesimal bending of knots is presented.

\begin{description} \frenchspacing \itemsep=-1pt \item[] Mathematics Subject Classification 2000: 53A04, 53C45, 57M25, 57M27, 78A25
\item[] Key words: infinitesimal bending, variation,  knot, Willmore energy, M\" obius
 energy, OpenGL, C++
\end{description}

\end{abstract}

%

\def\demo{%
  \par\topsep6pt plus6pt
  \trivlist
  \item[\hskip\labelsep\it Proof.]\ignorespaces}
\def\enddemo{\qed \endtrivlist}
\expandafter\let\csname enddemo*\endcsname=\enddemo

\def\qedsymbol{\ifmmode\bgroup\else$\bgroup\aftergroup$\fi
  \vcenter{\hrule\hbox{\vrule height.6em\kern.6em\vrule}\hrule}\egroup}
\def\qed{\ifmmode\else\unskip\nobreak\fi\quad\qedsymbol}

 \newtheorem{proof}{Proof}[section]
 \newtheorem{theorem}{Theorem}[section]
 \newtheorem{prop}{Proposition}[section]
 \newtheorem{lem}{Lemma}[section]
 \newtheorem{po}{Corollary}[section]
 \newtheorem{corollary}{Corollary}[section]
\newtheorem{lema}{Lemma}[section]
\newtheorem{pr}{Example}[section]
\newtheorem{rem}{Remark}[section]
 \newtheorem{definition}{Definition}[section]
 \newtheorem{algo}{Algorithm}[section]
 \newcounter{slicica}[section]
\newcommand{\s}{\par\refstepcounter{slicica}{\small Figure \thesection.\arabic{slicica}.}}
\renewcommand{\theslicica}{Figure \thesection.\arabic{slicica}.}
\newcommand{\ra}[1]{\textrm{rank}({#1})}  
\newcommand{\ind}[1]{\textrm{ind}({#1})}  
\newcommand{\cond}[1]{\textrm{cond}({#1})}  
\newcommand{\tr}[1]{\textrm{Tr}({#1})}  

\font\bigbf=cmbx12 \font\ssr=cmss8 \font\sst=cmtt8 \font\tt=cmtt10
\font\ssb=cmbx8 \font\srm=cmr8 \font\ssi=cmti8 \font\sar=cmss12
\font\srs=cmr6

\renewcommand{\theequation}{\thesection.\arabic{equation}}

\section{Introduction}

This paper is devoted to studying the shape descriptors of knots
during infinitesimal bending. The problem of infinitesimal bending
of knots is a special part of the theory of  deformation. Bending
theory considers bending of manifolds, isometrical deformations as
well as infinitesimal bending. It  requires use of differential
geometry, mechanics, physics and has applications in modern computer
graphics. Infinitesimal bending is "almost" an isometric
deformation, or it is an isometric deformation in a precise
approximation. Arc length is stationary under infinitesimal bending
with a given precision.
 The infinitesimal bending caught the attention of many of the brightest minds in the history of
mathematics, including A. D. Alexandrov \cite{AD}, W. Blaschke
\cite{Blaske}, A. Cauchy \cite{Cauchy},  S. Cohn-Vossen \cite{K-F},
R. Connelly \cite{Connelly}, N. V. Efimov \cite{Ef}, V. T. Fomenko
\cite{Fomenko2}, A. V. Pogorelov \cite{Pogorelov}, I. N. Vekua
\cite{Vekua},  I. Ivanova-Karatopraklieva and I. Kh. Sabitov
\cite{IK-S}.

The physical properties of knotted and linked configurations in
space have long been of interest to mathematicians. More recently,
these properties have become significant to biologists, physicists,
and engineers among others. Their depth of importance and breadth of
application are now widely appreciated and steady progress continues
to be made each year.

 There are a few major fields in applied knot theory: physical knot theory, knot theory in the life sciences, computational knot theory, and geometric knot theory.

Physical knot theory is the study of mathematical models of knotting phenomena, often motivated by considerations from biology, chemistry, and physics (Kauffman 1991)\cite{KHL1}. Physical knot theory is used to study how geometric and topological characteristics of filamentary structures, such as magnetic flux tubes, vortex filaments, polymers, DNA's, influence their physical properties and functions. It has applications in various fields of science, including topological fluid dynamics, structural complexity analysis and DNA biology (Kauffman 1991 \cite{KHL2}, Ricca 1998, \cite{RR1}).

Traditional knot theory models a knot as a simple closed loop in three-dimensional space. Such a knot has no
thickness or physical properties such as tension or friction. Physical knot theory incorporates more realistic models.
 The traditional model is also studied but with an eye toward properties of specific embeddings ("conformations") of the circle.
  Such properties include ropelength and various knot energies (O'Hara 2003 \cite{OH1}).

The most abstract connection between knots and science is the
phenomenon we might call "patterns of analysis", where purely
mathematical definitions and relationships in abstract knot theory
are echoed by definitions and relationships in physics. This is a
connection developed by L. Kauffman \cite{KHL1, KHL2} and others
such as   V. Jones \cite{Jones}, V. Turaev and N. Reshetikin
\cite{Resh}, E. Witten \cite{Witten}.
\smallskip

\begin{figure}[th]
  \centerline{
  \includegraphics[width=4.6truecm]{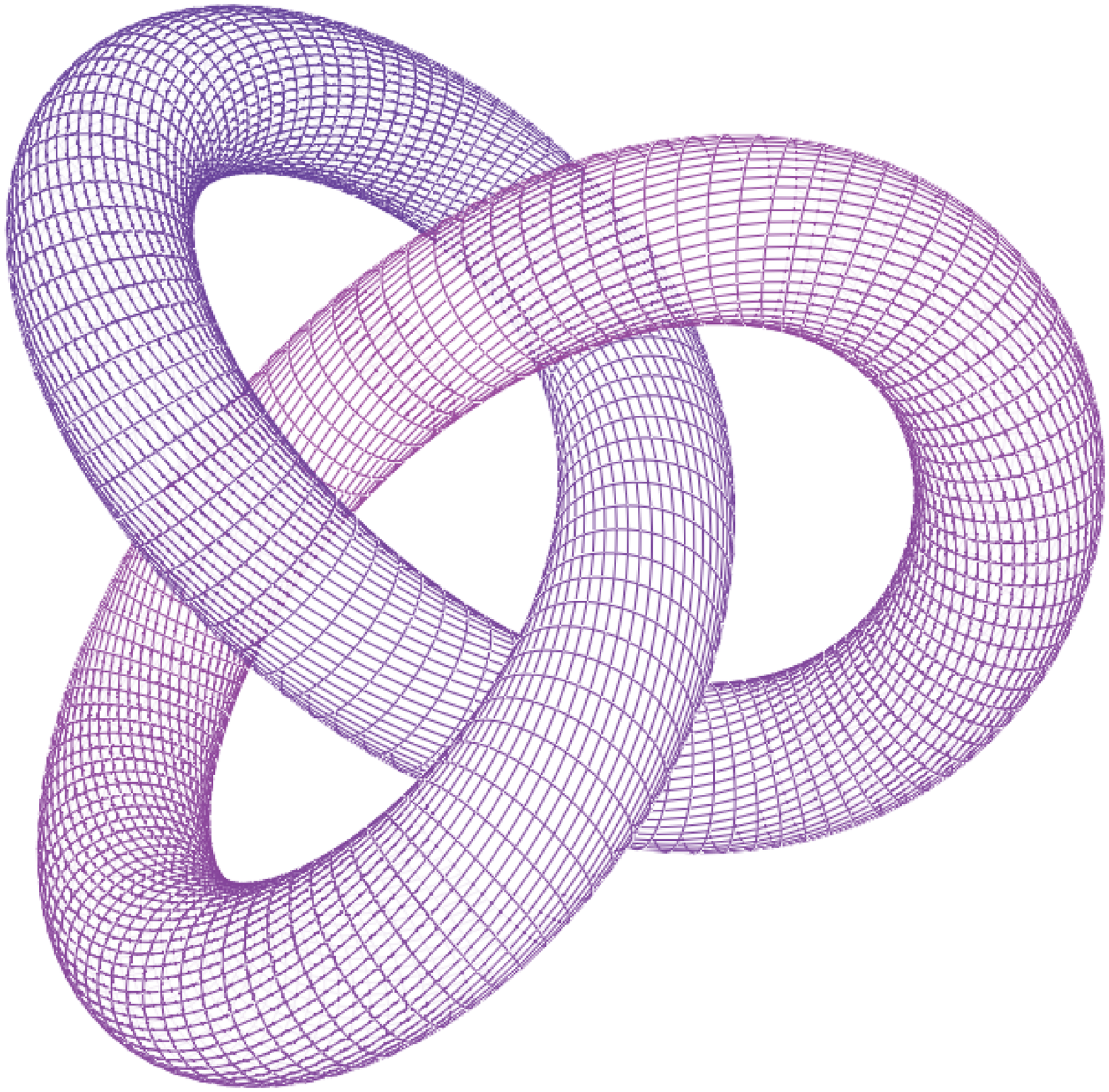}
  \includegraphics[width=4.6truecm]{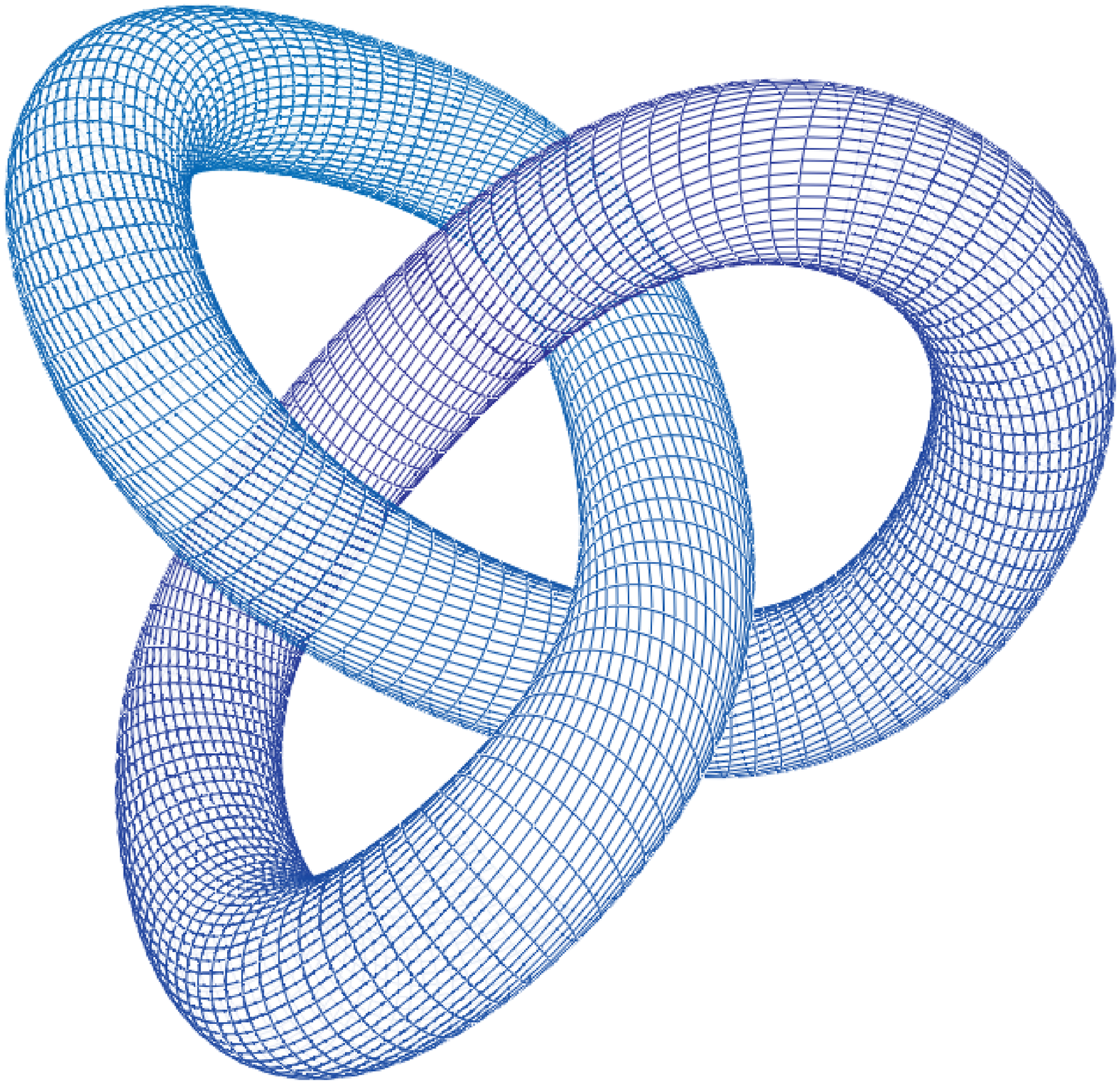}
  }
  \vspace*{8pt}
  \caption{Trefoil knot: basic and infinitesimally bent with $\epsilon = 0.3$.}\label{TrefoilFirst}
\end{figure}
\medskip
\begin{figure}[th]
   \centerline{
   \includegraphics[width=4.6truecm]{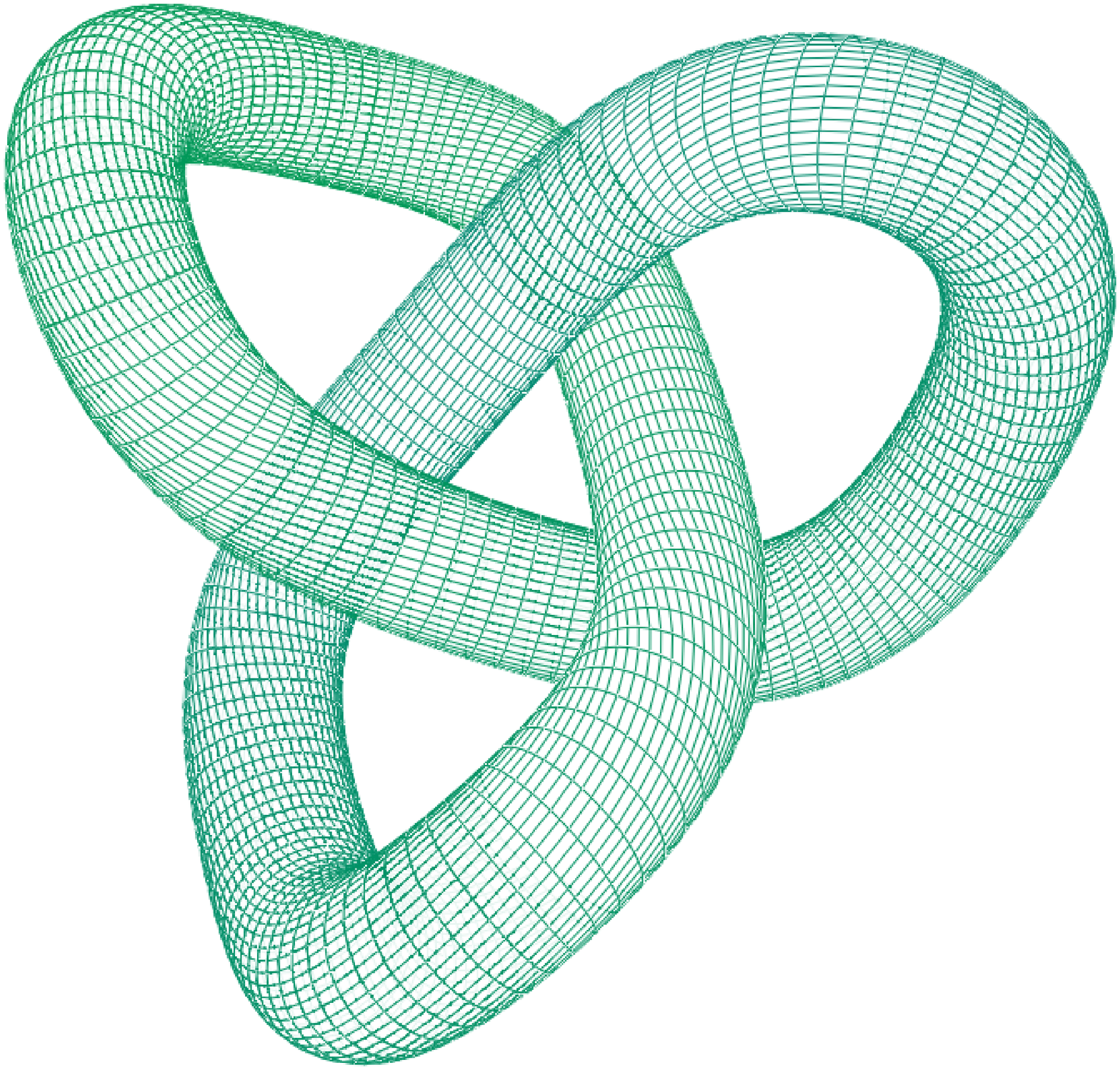}
   \includegraphics[width=4.6truecm]{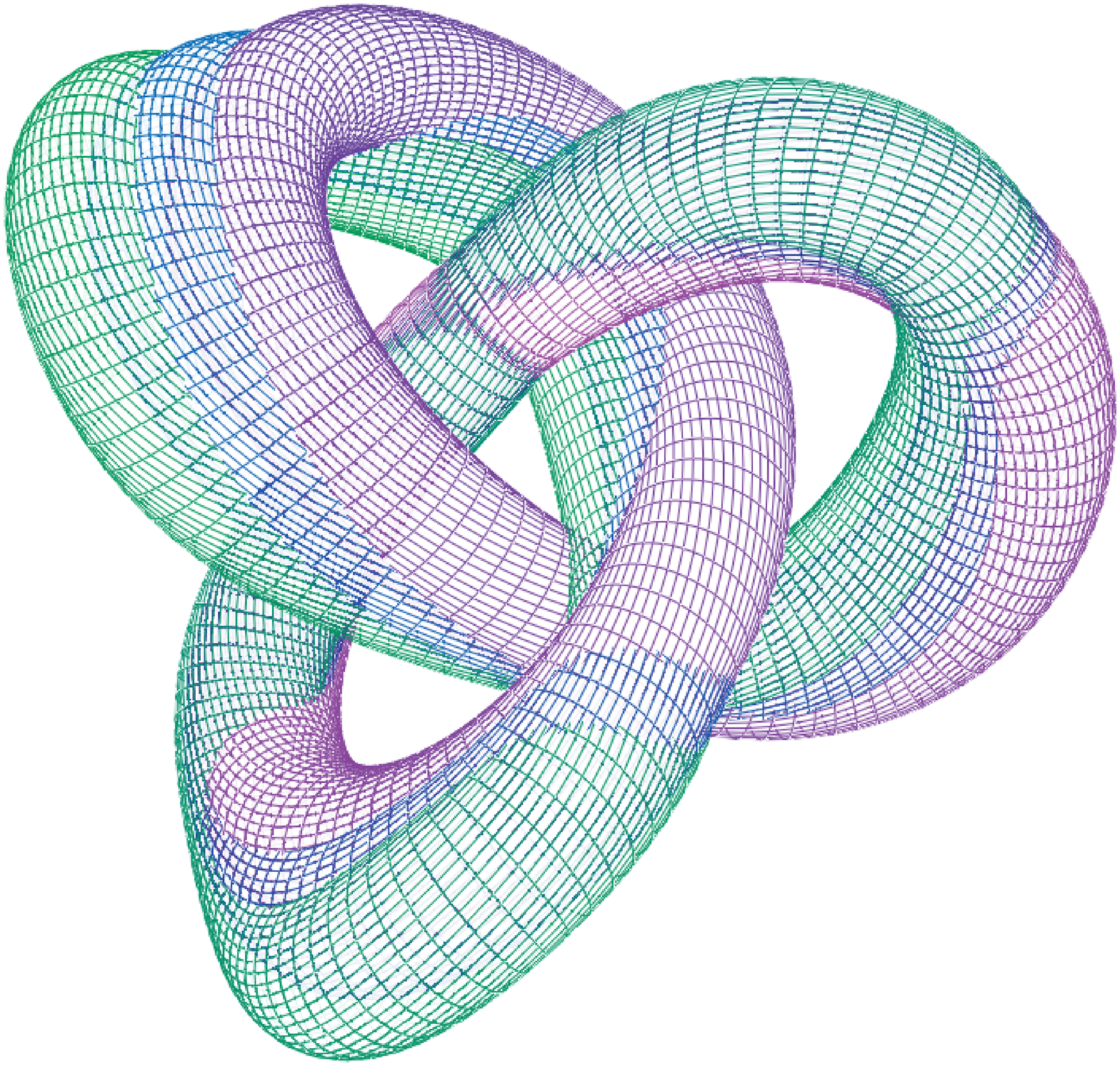}
   }
   \vspace*{8pt}
   \caption{Trefoil knot: infinitesimally bent with $\epsilon = 0.6$. and all knots together}\label{TrefoilSecond}
\end{figure}

Topologically ambient isotopic knots are by definition considered
equivalent. It is therefore  possible to think of a knot as a curve
with small but positive thickness which allows us to present it as a
tube. But in geometrical sense we could observe small deformations
of knots as a family of different curves.

Infinitesimal bending of a manifold, and particularly of a surface
or a curve in Euclidean 3-space, is a type of deformation
characterized by the rigidity of the arc length with a given
precision. In this case one observes changes of other magnitudes,
and then we say that they are rigid or flexible. For example, the
coefficients of the first quadratic form are rigid and of the second
one are flexible in the infinitesimal bending of a surface. The
theory of infinitesimal bending is in close connection with thin
elastic shell theory and leads to major mechanical applications.
Infinitesimal bending of curves and surfaces is studied, for
instance, in  (A. D. Aleksandrov 1936) \cite{AD}, (N. V. Efimov 1948)
\cite{Ef}, M. Najdanovi\' c \cite{Naj, NajVel}, I. Vekua
\cite{Vekua}, Velimirovi\' c et al \cite{LjV3}-\cite{LMN}.
Infinitesimal bending in generalized Riemannian space was studied at
Velimirovi\' c et al \cite{LSM}.

The paper is organized as follows: in Section 2, preliminary results
and notation regarding infinitesimal bending of curves are
presented. In Section 3, knot Willmore energy change under
infinitesimal bending is studied and the corresponding variation is
determined. In Section 4, the change of the M\" obius
 energy under infinitesimal bending of knots is considered. In
 Section 5, our visualization tool  devoted to visual representation
 of infinitesimal bending of knots is presented. The trefoil knot
 under infinitesimal bending is
 visualized.

 \setcounter{equation}{0}

\section{Infinitesimal bending of curves-preliminaries}

Let us consider a continuous biregular curve
\begin{eqnarray}\label{C}
C:{\bf r}={\bf r}(u),\quad u\in {\cal I}\subseteq{\cal
R}\end{eqnarray} included in a family of the curves
\begin{eqnarray}
C_\epsilon:\tilde{\bf r}(u,\epsilon)={\bf r}_\epsilon(u)={\bf
r}(u)+\epsilon{\bf z}(u),\quad u\in {\cal I},\,\epsilon\geq 0, \,
\epsilon\rightarrow 0,\end{eqnarray} where $u$ is a real parameter
and we get $C$ for $\epsilon=0$ ($C=C_0$).
\begin{definition}\cite{Ef}\label{def} A family of curves $C_\epsilon$ is
{\bf an infinitesimal bending of a curve } $C$ if
\begin{eqnarray}
ds_\epsilon^2-ds^2=o(\epsilon),\end{eqnarray} where ${\bf z}={\bf
z}(u)$ , ${\bf z}\in C^1$ is {\bf the infinitesimal bending field}
of the curve $C$.
\end{definition}

\begin{theorem}\emph{\cite{Ef}} A necessary and sufficient condition that ${\bf z}(u)$
is the infinitesimal bending field of a curve $C$ is to be
\begin{eqnarray}\label{potdov}d{\bf r}\cdot d{\bf z}=0,\end{eqnarray}
where $\cdot$ stands for the scalar product in ${\cal
R}^3$.
\end{theorem}
\begin{proof} According to the definition of the
infinitesimal bending of a curve $C$ the following holds
$$\aligned ds_\epsilon^2&-ds^2=d{\bf r}_\epsilon^2-d{\bf r}^2=d{\bf
r}^2+2\epsilon\, d{\bf r}\cdot d{\bf z}+\epsilon^2d{\bf z}^2-d{\bf
r}^2=2\epsilon \,d{\bf r}\cdot d{\bf z}+\epsilon^2d{\bf
z}^2=o(\epsilon)\\&\Leftrightarrow d{\bf r}\cdot d{\bf
z}=0.\endaligned$$ \
\end{proof}

\begin{theorem}\emph{\cite{Naj}} Under infinitesimal bending of the curves each line
element undertakes a non-negative addition, which is the
infinitesimal value of order at least 2 (in $\epsilon$), i. e.
\begin{equation}\label{1}
ds_\epsilon-ds=o(\varepsilon)\geq 0.\end{equation}
\end{theorem}
\begin{proof} As $$d{\bf r}=\dot{\bf r}(u)du,\quad d{\bf
z}=\dot{\bf z}(u)du,$$ according to (\ref{potdov}), for
infinitesimal bending field of a curve $C$ we have
\begin{equation}\label{nv16}\dot{\bf r}(u)\cdot\dot{\bf z}(u)=0,\end{equation} where
dot denotes derivative with respect to $u$. Based on that we have
\begin{equation}\aligned ds_\epsilon&=\|\dot{\bf r}_\epsilon(u)\|\,du=\|\dot{\bf
r}(u)+\epsilon\dot{\bf z}(u)\|\,du=(\|\dot{\bf
r}(u)\|^2+\epsilon^2\|\dot{\bf z}(u)\|^2)^{\frac
12}\,du\\&=\|\dot{\bf r}(u)\|\Big(1+\epsilon^2\frac{\|\dot{\bf
z}(u)\|^2}{\|\dot{\bf r}(u)\|^2}\Big)^{\frac
12}\,du=ds\Big(1+\epsilon^2\frac{\|\dot{\bf z}(u)\|^2}{\|\dot{\bf
r}(u)\|^2}\Big)^{\frac 12}\endaligned\end{equation} After using of
Maclaurin formula we get
$$ds_\epsilon=ds\Big(1+\epsilon^2\frac{\|\dot{\bf z}(u)\|^2}{2\|\dot{\bf
r}(u)\|^2}-\epsilon^4\frac{\|\dot{\bf z}(u)\|^4}{8\|\dot{\bf
r}(u)\|^4}+\ldots\Big)$$ i.e.
$$ds_\epsilon-ds=\epsilon^2\frac{\|\dot{\bf z}(u)\|^2}{2\|\dot{\bf
r}(u)\|^2}\,ds-\ldots,$$ which leads to
(\ref{1}). \
\end{proof}

The next theorem is related to determination of the infinitesimal
bending field of a curve $C$.

\begin{theorem}\emph{\cite{LjV3}}
The infinitesimal bending field for the curve $C$ (\ref{C}) reads
\begin{equation}\label{poljez11}
{\bf z}(u)=\int[p(u){\bf n}_1(u)+q(u){\bf
n}_2(u)]\,du,\end{equation} where $p(u)$ and $q(u)$ are arbitrary
integrable functions, and vectors ${\bf n}_1(u)$ and ${\bf n}_2(u)$
are respectively unit principal normal and binormal vector fields of
a curve $C$.
\end{theorem}
\begin{proof} According to (\ref{nv16}) we have
\begin{equation}
\dot{\bf r}\cdot\dot{\bf z}=0,\quad {\mbox{ i. e. }}\dot{\bf
r}\bot\dot{\bf z}.
\end{equation}
Based on that we conclude that $\dot{\bf z}$ lies in the normal
plane of the curve $C$, i. e.
\begin{equation}\label{1.7}
\dot{\bf z}(u)=p(u){\bf n}_1(u)+q(u){\bf n}_2(u),
\end{equation}
where $p(u)$ and $q(u)$ are arbitrary integrable functions.
Integrating (\ref{1.7}) we obtain
(\ref{poljez11}). \
\end{proof}

As
\begin{equation}\label{e8.1.8}
 {\bf{n}}_1=\frac{(\dot{\bf{r}}\cdot\dot{\bf{r}})\ddot{\bf{r}}-
(\dot{\bf{r}}\cdot\ddot{\bf{r}})\dot{\bf{r}}} {\|\dot{\bf{r}}\|
\|\dot{\bf{r}}\times\ddot{\bf{r}}\|}   ,\quad {\bf{n}}_2=
\frac{\dot{\bf{r}}\times\ddot{\bf{r}}}
{\|{\dot{\bf{r}}\times\ddot{\bf{r}}}\|},
\end{equation}
infinitesimal bending field can be written in the form

$${\bf{z}}(u) =\int[p(u)
\frac{(\dot{\bf{r}}\cdot\dot{\bf{r}})\ddot{\bf{r}}-
(\dot{\bf{r}}\cdot\ddot{\bf{r}})\dot{\bf{r}}} {\|\dot{\bf{r}}\|
\|\dot{\bf{r}}\times\ddot{\bf{r}}\|}+
q(u)\frac{\dot{\bf{r}}\times\ddot{\bf{r}}}
{\|{\dot{\bf{r}}\times\ddot{\bf{r}}}\|}]du$$ where $p(u)$ and $
q(u)$ are arbitrary integrable functions, or in the form
\begin{equation}\label{e8.1.9}
    {\bf{z}}(u)=\int[P_1(u)\dot{\bf{r}}+P_2(u)\ddot{\bf{r}}+
Q(u)(\dot{\bf{r}}\times\ddot{\bf{r}})]du
\end{equation}
where $P_i(u),\; i=1, 2,\;\text{i}\;Q(u)$ are arbitrary integrable
functions, too.

Under an infinitesimal bending, geometric magnitudes of the curve
are changed which is described with variations of these geometric
magnitudes.

\begin{definition}\cite{Vekua}\label{defvar} Let  ${\cal A}={\cal A}(u)$ be a magnitude which characterizes a geometric property on the curve  $C$
and   ${\cal A}_\epsilon={{\cal A}}_\epsilon(u)$  the corresponding
magnitude on the curve $C_\epsilon$ being infinitesimal bending of
the curve $C$, and set
\begin{eqnarray}\label{dvar}
\Delta {\cal A}\!=\!{\cal A}_\epsilon-{\cal A}\!=\!\epsilon\,\delta
{\cal A}+\epsilon^2\,\delta^2 {\cal A}+\ldots+ \epsilon^n\,\delta^n
{\cal A}+\ldots\end{eqnarray} The coefficients $\delta {\cal
A},\delta^2 {\cal A},\ldots,\delta^n {\cal A},\ldots$ are {\bf the
first, the second, ..., the n-th variation} of the geometric
magnitude ${\cal A}$, respectively under the infinitesimal bending
$C_\epsilon$ of the curve $C$.
\end{definition}
\smallskip

Let us mark some properties of the variations according to
\cite{NajVel}:

\smallskip
 {\bf I.} For the variations of the product of  geometric magnitudes
 it
 is effective the equation
\begin{equation}\label{formulaAB}\delta^n{\cal AB}=\sum_{i=0}^n\delta^i{\cal
A}\,\delta^{n-i}{\cal B},\quad n\geq 0, \quad(\delta^0{\cal
A}\overset{def}{=} {\cal A}).\end{equation}

According to Def. (\ref{defvar}), the variations of  geometric
magnitudes ${\cal A}$ and ${\cal B}$, as well as of the product
${\cal AB}$ are:
\begin{equation}\label{varA}
\Delta {\cal A}\!=\!{\cal A}_\epsilon-{\cal A}\!=\!\epsilon\,\delta
{\cal A}+\epsilon^2\,\delta^2 {\cal A}+\ldots +\epsilon^n\,\delta^n
{\cal A}+\ldots\
\end{equation}
\begin{equation}\label{varB}
\Delta {\cal B}\!=\!{\cal B}_\epsilon-{\cal B}\!=\!\epsilon\,\delta
{\cal B}+\epsilon^2\,\delta^2 {\cal B}+\ldots+ \epsilon^n\,\delta^n
{\cal B}+\ldots\
\end{equation}
\begin{equation}\label{varAB}
\Delta {\cal AB}\!=\!{\cal A}_\epsilon{\cal B}_\epsilon-{\cal
AB}\!=\!\epsilon\,\delta ({\cal AB})+\epsilon^2\,\delta^2 ({\cal
AB})+\ldots +\epsilon^n\,\delta^n ({\cal AB})+\ldots\
\end{equation}
respectively. On the other hand, the following equality is
satisfied:
\begin{equation}\label{varAB2}
\Delta {\cal AB}\!=\!{\cal A}_\epsilon{\cal B}_\epsilon-{\cal
AB}\!=\!{\cal A}_\epsilon{\cal B}_\epsilon-{\cal A}_\epsilon B+{\cal
A}_\epsilon B-{\cal AB}\!=\!{\cal A}_\epsilon({\cal
B}_\epsilon-B)+({\cal A}_\epsilon-A)B
\end{equation}
Substituting (\ref{varA}) and (\ref{varB}) into equation
(\ref{varAB2}) we obtain
\begin{equation}
\Delta {\cal AB}\!=\!\epsilon(A\delta B+B\delta
A)+\epsilon^2(A\delta^2B+\delta A\,\delta
B+B\delta^2A)+\ldots+\epsilon^n\sum_{i=0}^n\delta^i{\cal
A}\,\delta^{n-i}{\cal B}+\ldots
\end{equation}
As the left sides of the last equation and of the equation
(\ref{varAB}) are equal, the equation (\ref{formulaAB}) is a valid
one.

\smallskip

{\bf II.} An arbitrary order variation of a derivative  is the
derivative of the variation, i. e.
\begin{equation}\label{varpart}
\delta^n(\frac{d {\cal A}}{d u})=\frac{d(\delta^n{\cal A})}{d
u},\quad n\geq 0.
\end{equation}
For
\begin{equation}
\frac{d {\cal A}}{d u}={\cal B},
\end{equation}
 using ({\ref{dvar}) we obtain the equation
\begin{equation}\label{parta}
\triangle\frac{d{\cal A}}{d u}\!=\!\triangle{\cal
B}\!=\!\epsilon\,\delta {\cal B}+\epsilon^2\,\delta^2 {\cal
B}+\ldots+ \epsilon^n\,\delta^n {\cal B}+\ldots\!=\!\epsilon\,\delta
\frac{d{\cal A}}{d u}+\epsilon^2\,\delta^2 \frac{d{\cal A}}{d
u}+\ldots+ \epsilon^n\,\delta^n \frac{d{\cal A}}{d u}+\ldots\
\end{equation}
We also have
\begin{equation}\label{partA2}\aligned
\triangle\frac{d{\cal A}}{d u}\!&=\!\triangle{\cal B}\!=\!{\cal
B}_\epsilon(u)-{\cal B}(u)\!=\!\frac{d{\cal
A}_\epsilon(u)}{d u}-\frac{d{\cal A}(u)}{d u}\\
\!&=\!\frac{d}{d u}[{\cal A}_\epsilon(u)-{\cal
A}(u)]\!=\!\frac{d\triangle{\cal A}}{d
u}\!=\!\frac{d(\epsilon\,\delta {\cal A}+\epsilon^2\,\delta^2 {\cal
A}+\ldots+ \epsilon^n\,\delta^n
{\cal A}+\ldots)}{d u}\\
\!&=\!\epsilon\frac{d(\delta {\cal A})}{d
u}+\epsilon^2\frac{d(\delta^2 {\cal A})}{d
u}+\ldots+\epsilon^n\frac{d(\delta^n {\cal A})}{d u}+\ldots
\endaligned\end{equation}
By comparing the equations (\ref{parta}) and (\ref{partA2}), we
confirm validity of the equation (\ref{varpart}). The same case is
for the differential, i. e.

\smallskip
{\bf III.} $\delta^n(d{\cal A})=d(\delta^n{\cal A}), \quad n\geq 0.$

\smallskip
 In the case when we consider only the first variations we
can represent the magnitude ${\cal A}_\epsilon$ as

$${\cal A}_\epsilon={\cal A}+\epsilon\,\delta {\cal A},$$ by neglecting the  $\epsilon^n$-terms, $n\geq 2$.

The first variation is of course given by
\begin{equation}\label{prvavarijacija}
\delta {\cal A}\!=\!\frac{d}{d\epsilon}{\cal
A}_\epsilon(u)\big|_{\epsilon=0},
\end{equation}
i. e.
\begin{equation}
\delta {\cal A}=\lim_{\epsilon\rightarrow 0}\frac{\Delta{\cal
A}}{\epsilon}= \lim_{\epsilon\rightarrow 0}\frac{{\cal
A}_\epsilon(u)-{\cal A}(u)}{\epsilon}.\end{equation}

Also, we have:
 \begin{equation}\label{pravila}
 {\bf Ia.}\quad \delta({\cal AB})={\cal A}\delta {\cal B}+{\cal B}\delta {\cal A},\quad {\bf IIa.}\quad
 \delta\Big(\frac{\partial {\cal A}}{\partial u}\Big)=\frac{\partial(\delta
 {\cal A})}{\partial u},\quad {\bf IIIa.}\quad\delta(d{\cal A})=d(\delta{\cal A}).\end{equation}

 According to (\ref{1}) we have
 $$\triangle
 (ds)=0\cdot\epsilon+\epsilon^2\delta^2(ds)+...=\delta(ds)\epsilon+\delta^2(ds)\epsilon^2+...$$
 thus
 \begin{equation}
 \delta(ds)=0.
 \end{equation}
 Therefore, under infinitesimal bending of a curve, the first variation of
 the line element $ds$ is equal to zero.

 {\bf A curve parameterized by the arc length under
infinitesimal bending.} Let us consider a curve
\begin{equation}\label{cs}C: {\bf r}={\bf r}(s)=r[u(s)],\quad s\in
{\cal J}=[0,L],\end{equation}parameterized by the arc length $s$.
The unit tangent to the curve is given by ${\bf t}={\bf r}'$, where
prime denotes a derivative with respect to arc length $s$. Clearly,
${\bf t}'$ is orthogonal to ${\bf t}$, but $ {\bf t}''$ is not. The
classical Frenet equations
\begin{equation}\label{frenet}\aligned
{\bf t}'&=k{\bf n}_1,\\
{\bf n}_1'&=-k{\bf t}+\tau{\bf n}_2,\\
{\bf n}_2'&=-\tau{\bf n}_1,\endaligned\end{equation} describe the
construction of an orthonormal basis $\{\bf{t},{\bf n}_1,{\bf
n}_2\}$ along a curve, where ${\bf n}_1$ and ${\bf n}_2$ are
respectively unit principal normal and binormal vector fields of the
curve. Note that for the Frenet trihedron to be well-defined, we
need $k\neq 0$ throughout. We choose an orientation with ${\bf
n}_2={\bf t}\times{\bf n}_1$. $k$ and $\tau$ are respectively the
curvature and the torsion.

Let us consider an infinitesimal bending of the curve (\ref{cs}),
\begin{equation}\label{ibs}
C_\epsilon: \tilde{\bf r}(s,\epsilon)={\bf r}_\epsilon(s)={\bf
r}(s)+\epsilon{\bf z}(s).\end{equation} As the vector field ${\bf
z}$ is defined in the points of the curve (\ref{cs}), it can be
presented in the form
\begin{equation}\label{poljez}
{\bf z}=z{\bf t}+z_1{\bf n}_1+z_2{\bf n}_2,
\end{equation}
where $z{\bf t}$ is tangential and $z_1{\bf n}_1+z_2{\bf n}_2$ is
normal component, $z,z_1,z_2$ are the functions of $s$.

\begin{theorem}\emph{\cite{Naj}}
Necessary and sufficient condition for the field ${\bf z}$
(\ref{poljez}) to be infinitesimal bending field of the curve $C$
(\ref{cs}) is
\begin{equation}\label{potdovz}
z'-kz_1=0, \end{equation} where $k$ is the curvature of $C$.
\end{theorem}
\begin{proof} According to (\ref{potdov}), the necessary and
sufficient condition for the field ${\bf z}$ to be infinitesimal
bending field of the curve $C$ is \begin{equation}\label{rzs}{\bf
r}'\cdot{\bf z}'=0,\end{equation} i. e. ${\bf t}\cdot{\bf z}'=0$.
Substituting Eq. (\ref{poljez}) into the previous equation and using
Frenet equations (\ref{frenet}), we obtain
(\ref{potdovz}). \
\end{proof}

 Let us describe the behavior of some geometric magnitudes
under
 infinitesimal bending of a curve according to \cite{Naj}.
As it is $\delta{\bf t}=\delta{\bf r}'=(\delta{\bf r})'={\bf z}'$,
using (\ref{poljez}), (\ref{potdovz}) and Frenet equations we obtain
\begin{equation}\label{deltat}
\delta{\bf t}=(z_1'-\tau z_2+kz)\,{\bf n}_1+(z_2'+\tau z_1)\,{\bf
n}_2.
\end{equation}

Applying commutativity of the variation and the derivative, we have
$\delta{\bf t}'=(\delta{\bf t})'$. Based on (\ref{deltat}), Frenet
equations and $z'=kz_1$ (due to (\ref{potdovz})), one obtains
\begin{equation}\aligned\label{deltat'}
\delta{\bf t}'=&-k(kz+z_1'-\tau z_2)\,{\bf
t}+(k'z+z_1''+(k^2-\tau^2)z_1-2\tau z_2'-\tau'z_2)\,{\bf
n}_1\\&+(k\tau z+2\tau z_1'+\tau'z_1+z_2''-\tau^2z_2)\,{\bf n}_2.
\endaligned\end{equation}
To evaluate $\delta k$, we take a variation of the first equation in
(\ref{frenet}). We obtain
$$\delta {\bf t}'=\delta k\,{\bf n}_1+k\,\delta{\bf n}_1.$$
Dotting with ${\bf n}_1$ and using the fact that a unit vector of
the orthonormal basis and its variation are orthogonal, the previous
equation becomes $\delta k={\bf n}_1\cdot \delta{\bf t}'$. This
leads to \begin{equation}\label{deltak} \delta
k=k'z+z_1''+(k^2-\tau^2)z_1-2\tau z_2'-\tau'z_2.\end{equation}
 after
using (\ref{deltat'}).

Let us take a variation of the Frenet equation for ${\bf n}_1'$ and
dot with ${\bf n}_2$. We have
\begin{equation}\label{3.13}
\delta\tau=k{\bf n}_2\cdot\delta{\bf t}+{\bf n}_2\cdot\delta{\bf
n}_1'.
\end{equation}
We now rewrite the second term on the right hand side as
\begin{equation}\label{3.14}
{\bf n}_2\cdot\delta{\bf n}_1'=({\bf n}_2\cdot\delta{\bf n}_1)'-{\bf
n}_2'\cdot\delta{\bf n}_1=({\bf n}_2\cdot\delta{\bf
n}_1)',\end{equation} after using the third Frenet equation. As it
is ${\bf r}''=k{\bf n}_1$, we have ${\bf t}'=k{\bf n}_1$, i. e.
${\bf n}_1'=\frac{1}{k}\,{\bf t}'.$ Farther, $\delta{\bf
n}_1=\delta(\frac{1}{k})\,{\bf t}'+\frac{1}{k}\,\delta{\bf t}',$
\begin{equation}\label{3.15}
{\bf n}_2\cdot\delta{\bf n}_1={\bf n}_2\cdot
\big[\delta(\frac{1}{k})\,k{\bf n}_1+\frac{1}{k}\,\delta{\bf
t}'\big]=\frac{1}{k}{\bf n}_2\cdot\delta{\bf t}'. \end{equation}
From (\ref{3.13}), (\ref{3.14}) and (\ref{3.15}) we obtain
\begin{equation}\label{3.16}
\delta\tau=k{\bf n}_2\cdot\delta{\bf t}+\Big(\frac{1}{k}{\bf
n}_2\cdot\delta{\bf t}'\Big)'\end{equation} Substituting
(\ref{deltat}) and (\ref{deltat'}) into (\ref{3.16}) and using
(\ref{potdovz}) we obtain
\begin{equation}\label{deltator}
\delta\tau=z\tau'+k(z_2'+2\tau z_1)+\Big\{\frac{1}{k}\big[2\tau
z_1'+\tau'z_1+z_2''-\tau^2z_2\big]\Big\}'.
\end{equation}

Applying the known roles about the variation, Frenet formulas, the
facts that ${\bf n}_1=\frac{{\bf t}'}{k}$ and ${\bf n}_2={\bf
t}\times{\bf n}_1$, we simply get the variations of the normals:

\begin{equation}
\delta{\bf n}_1=-(kz+z_1'-\tau z_2){\bf t}+\frac{1}{k}(k\tau
z+z_2''-\tau^2z_2+2\tau z_1'+\tau'z_1){\bf n}_2,
\end{equation}
\begin{equation}
\delta{\bf n}_2=-(z_2'+\tau z_1){\bf t}-\frac{1}{k}(k\tau
z+z_2''-\tau^2z_2+2\tau z_1'+\tau'z_1){\bf n}_1.
\end{equation}

Based on the previous considerations,  corresponding geometric
magnitudes of deformed curves under infinitesimal bending are:
$$
\tilde{k}=k_\epsilon =k+\epsilon[k'z+z_1''+(k^2-\tau^2)z_1-2\tau
z_2'-\tau'z_2],$$ $$ \tilde{\tau}=\tau_\epsilon
=\tau+\epsilon\Big\{z\tau'+k(z_2'+2\tau
z_1)+\Big[\frac{1}{k}\big(2\tau
z_1'+\tau'z_1+z_2''-\tau^2z_2\big)\Big]'\Big\},$$ $$ \tilde{{\bf
t}}={\bf t}_\epsilon={\bf t}+\epsilon\big[(z_1'-\tau z_2+kz)\,{\bf
n}_1+(z_2'+\tau z_1)\,{\bf n}_2\big],$$ $$ \tilde{{\bf n}_1}={\bf
n}_{1\epsilon}={\bf n}_1+\epsilon\big[-(kz+z_1'-\tau z_2){\bf
t}+\frac{1}{k}(k\tau z+z_2''-\tau^2z_2+2\tau z_1'+\tau'z_1){\bf
n}_2\big],$$ $$ \tilde{{\bf n}_2}={\bf n}_{2\epsilon}={\bf
n}_2+\epsilon\big[-(z_2'+\tau z_1){\bf t}-\frac{1}{k}(k\tau
z+z_2''-\tau^2z_2+2\tau z_1'+\tau'z_1){\bf n}_1\big],
$$
after neglecting the terms with $\epsilon^n,\quad  n\geq 2$.

 \setcounter{equation}{0}
\section{Knot Willmore energy change under infinitesimal bending }

Curvature-based energies play a main role in the description of both
physical and non-physical systems (see \cite{cap, anna, Hel, LMM,
LMN}).

Let us observe the Willmore energy of a knot $K$:
\begin{equation}
{\cal W}=\frac 12\int_{\cal J}k^2\,ds.
\end{equation}
The Willmore energy of a deformed knot will be
\begin{equation}
{\cal W}_\epsilon=\frac 12\int_{\cal
J}k_{\epsilon}^2\,ds_\epsilon=\frac 12\int_{\cal
J}(k+\epsilon\,\delta k)^2[ds+\epsilon\,\delta(ds)],\end{equation}
i. e.
\begin{equation} {\cal W}_\epsilon={\cal
W}+\epsilon\big[\int_{\cal J}k\,\delta k\,ds+\frac 12\int_{\cal
J}k^2\,\delta(ds)\big].
\end{equation}
As it is $\delta(ds)=0$, we obtain that
\begin{equation} \delta{\cal W}=\int_{\cal J}k\,\delta k\,ds,
\end{equation}
Applying (\ref{deltak}) we get
\begin{equation}\label{39}
\delta{\cal W}=\int_{\cal J}k[k'z+z_1''+(k^2-\tau^2)z_1-2\tau
z_2'-\tau'z_2]\,ds,
\end{equation}
wherefrom we have the next equation
\begin{equation}\label{varwill}\aligned
\delta{\cal W}&=\int_{\cal J}ds\,\big[(k''+\frac
12k^3-k\tau^2)z_1+(2k'\tau+k\tau')z_2\big]\\&+\int_{\cal
J}ds\,\big[\frac 12k^2z-k'z_1+kz_1'-2k\tau
z_2\big]',\endaligned\end{equation}  after a bit of calculation.

The Willmore energy of a deformed knot under infinitesimal bending
is
\begin{equation}\aligned
{\cal W}_\epsilon &={\cal W}+\epsilon\Big\{\int_{\cal
J}ds\,\big[(k''+\frac
12k^3-k\tau^2)z_1+(2k'\tau+k\tau')z_2\big]\\&+\int_{\cal
J}ds\,\big[\frac 12k^2z-k'z_1+kz_1'-2k\tau
z_2\big]'\Big\}.\endaligned\end{equation}

 In the case of infinitesimal bending of knots we specify   the condition ${\bf z}(0)={\bf z}(L)$  for the
infinitesimal bending field in order to get a family of closed
curves. Also, we suppose that the knot, as well as the infinitesimal
bending field are sufficiently smooth. Keeping this in mind we have
the following theorem.

\begin{theorem} Under infinitesimal bending of a knot $K$,  variation of its
Willmore energy is
\begin{equation}
\delta{\cal W}=\int_{\cal J}ds\,\big[(k''+\frac
12k^3-k\tau^2)z_1+(2k'\tau+k\tau')z_2\big],\end{equation} where $k$
and $\tau$ are the curvature and the torsion of $K$, respectively.
\end{theorem}

 \setcounter{equation}{0}
\section{Knot M\" obius energy change under infinitesimal bending }

In mathematics, the M\" obius energy of a knot is a particular knot energy, i.e. a functional on the space of knots. It was discovered by Jun O'Hara, who demonstrated that the energy blows up as the knot's strands get close to one another. This is a useful property because it prevents self-intersection and ensures the result under gradient descent is of the same knot type.

We will here consider this type of energy under infinitesimal bending of knot.

Let $K$ be a smooth knot in 3-space ${\cal R}^3$. We can present $K$
as the image of the standard unit circle $S$ in the plane under a
smooth map ${\bf r}:S\rightarrow{\cal R}^3$. Then the M\" obius
energy of a knot is
\begin{equation*}
E(K)=\int_S\int_S\Big(\frac{1}{\|{\bf r}(s)-{\bf
r}(t)\|^2}-\frac{1}{l(s,t)^2}\Big)\,\|\dot{{\bf
r}}(s)\|\,\|\dot{{\bf r}}(t)\|\,dsdt,
\end{equation*}
where $l(s,t)$ denotes the minimum distance along the knot $K$
between points ${\bf r}(s)$ and ${\bf r}(t)$.

Let us assume the parametrization ${\bf r}$ be by arc length, then

\begin{equation}\label{rprim1}\|\dot{\bf r}(s)\|\equiv\|{\bf r}'(s)\|=1.\end{equation} So, we have

\begin{equation}
E(K)=\int_S\int_S\Big(\frac{1}{\|{\bf r}(s)-{\bf
r}(t)\|^2}-\frac{1}{l(s,t)^2}\Big)\,dsdt.
\end{equation}

Let us observe infinitesimal bending of $K$:
\begin{equation*}
K_\epsilon\,:\,{\bf r}_\epsilon={\bf r}(s)+\epsilon{\bf z}(s).
\end{equation*}
The energy of the deformed knot $K_\epsilon$ will be
\begin{equation*}
E(K_\epsilon)=\int_S\int_S\Big(\frac{1}{\|{\bf r}_{\epsilon}(s)-{\bf
r}_{\epsilon}(t)\|^2}-\frac{1}{l_{\epsilon}(s,t)^2}\Big)\,\|{{\bf
r}_{\epsilon}'}(s)\|\,\|{{\bf r}_{\epsilon}'}(t)\|\,dsdt,
\end{equation*}
where $l_{\epsilon}(s,t)$ is the distance long the deformed knot
$K_\epsilon$ between the points ${\bf r}_\epsilon(s)$ and ${\bf
r}_\epsilon(t)$.

Let us consider the following integrals:
\begin{equation}
E_1(K_\epsilon)=\int_S\int_S\frac{1}{\|{\bf r}_{\epsilon}(s)-{\bf
r}_{\epsilon}(t)\|^2}\,\|{{\bf r}}'(s)\|\,\|{{\bf
r}_{\epsilon}'}(t)\|\,dsdt,
\end{equation}
\begin{equation}\label{E2}
E_2(K_\epsilon)=\int_S\int_S\frac{1}{l_{\epsilon}(s,t)^2}\Big\|{{\bf
r}_{\epsilon}'}(s)\|\,\|{{\bf r}'}(t)\|\,dsdt.
\end{equation}
We are going to check the variations $\delta E_1$ and $\delta E_2$.
It holds
$$\aligned\|{{\bf r}}'_\epsilon(s)\|&=\|{{\bf r}}'(s)+\epsilon{{\bf
z}}'(s)\|=(\|{{\bf r}}'(s)+\epsilon{{\bf
z}}'(s)\|^2)^{1/2}\\&=(\|{{\bf r}}'(s)\|^2+2\epsilon{\bf
r}'(s)\cdot{\bf z}'(s)+\epsilon^2\|{{\bf
z}}'(s)\|^2)^{1/2}\endaligned$$ Using (\ref{rzs}), (\ref{rprim1})
and Maclaurin formula, we obtain
\begin{equation}\label{normrprim}
\|{{\bf r}}'_\epsilon(s)\|=1+\frac 12\epsilon^2\|{\bf
z}'(s)\|^2-\frac{1}{8}\epsilon^4\|{\bf z}'(s)\|^4+\ldots
\end{equation}
Therefore, $\delta\|{\bf r}'(s)\|=0$, i. e. $\|{{\bf
r}'}_\epsilon(s)\|=\|{\bf r}'(s)\|=1$ after neglecting the terms
with $\epsilon^n$, $n\geq2$. Further,

$$\aligned E_1(K_\epsilon)&=\int_S\int_S\frac{1}{\|{\bf r}_{\epsilon}(s)-{\bf
r}_{\epsilon}(t)\|^2}\,dsdt\\&=\int_S\int_S\frac{1}{\|({\bf
r}(s)-{\bf r}(t))+\epsilon({\bf z}(s)-{\bf
z}(t))\|^2}\,dsdt\\&=\int_S\int_S\frac{1}{\|{\bf r}(s)-{\bf
r}(t)\|^2}\Big(1-\frac{2\epsilon({\bf r}(s)-{\bf r}(t))\cdot({\bf
z}(s)-{\bf z}(t))}{\|{\bf r}(s)-{\bf
r}(t)\|^2}+\epsilon^2\ldots\Big)\,dsdt\endaligned$$ wherefrom we
obtain that
\begin{equation}\label{eq1}
\delta
E_1(K)=\frac{dE_1(K_\epsilon)}{ds}\Big|_{\epsilon=0}=-2\int_S\int_S\frac{({\bf
r}(s)-{\bf r}(t))\cdot({\bf z}(s)-{\bf z}(t))}{\|{\bf r}(s)-{\bf
r}(t)\|^4}\,dsdt
\end{equation}

On the other hand, we have
\begin{equation*}\aligned
l_\epsilon(s,t)&=\int_s^t\|\dot{{\bf
r}}_{\epsilon}(u)\|\,du=\int_s^t\|\dot{{\bf r}}(u)+\epsilon\dot{{\bf
z}}(u)\|\,du\\&=l(s,t)+\frac 12\epsilon^2\int_s^t\frac{\|\dot{{\bf
z}}(u)\|^2}{\|\dot{{\bf
r}}(u)\|}\,du-\frac{1}{8}\epsilon^4\int_s^t\frac{\|\dot{{\bf
z}}(u)\|^4}{\|\dot{{\bf r}}(u)\|^3}\,du+\ldots,
\endaligned
\end{equation*}
i. e. $\delta l(s,t)=0$ and, using (\ref{E2}) and
({\ref{normrprim}), we get
\begin{equation}\label{eq2}
\delta
E_2(K)=\frac{dE_2(K_\epsilon)}{d\epsilon}\Big|_{\epsilon=0}=0.
\end{equation}
According to (\ref{eq1}) and (\ref{eq2}) we obtain the next theorem.

\begin{theorem} Under infinitesimal bending of a knot $K$, variation of
the M\" obius energy is

\begin{equation}
\delta
E(K)=\frac{dE_1(K_\epsilon)}{ds}\Big|_{\epsilon=0}=2\int_S\int_S\frac{({\bf
r}(t)-{\bf r}(s))\cdot({\bf z}(s)-{\bf z}(t))}{\|{\bf r}(s)-{\bf
r}(t)\|^4}\,dsdt.
\end{equation}
\end{theorem}

\section{Knot infinitesimal bending - examples}

Here we will illustrate infinitesimally bending on knots, given bay a simple parametric representation.
Our aim is to visualize changing in shape and geometrical magnitudes when infinitesimal bending is applied.
We start from knot representations as a curve in ${\cal R}^3$.
Then, according to (\ref{poljez11}) we apply the bending described
by $p(u)$ and $q(u)$. The bending field, given by (\ref{poljez11})
is defined by an integral whose sub integral function include arbitrary functions $p(u)$ and $q(u)$.
The knot is visualized as polygonal line which connect points on them. At every such point, as well as, every subdivision point for the purpose of numerically integral calculation we should calculate functions: the curve, $p(u)$ and $q(u)$, first and second derivative and normals of the curve.
Those calculations are necessary to obtain transformed shape of the curve.
Instead of using existing software capable to do symbolic and numeric calculations, we decided to develop our own software tool in  {\it Microsoft} Visual C++.
The tool is aimed for manipulating explicitly defined functions,
starting from its usual symbolic definitions as a string. The second
step is parsing its symbolic definitions to obtain an internal,
tree like, form. For the purpose of efficiency we parse function once, then calculate it many
times. There are some important benefits of the tree like
form as: combine more functions to obtain a compound function like sub integral function for infinitesimal field, make derivatives. Our tool has not possibility to calculate integral symbolically, instead we are using ability for fast calculation of sub integral function  according to
$F(x)=\int_{0}^x f(x)dx$ , with possibility to add a integration
constant.

Visualization of the knot and obtaining 3D model is done by using
{\it OpenGL}. In the  figures Figs. \ref{TrefoilFirst}-\ref{EightSecond} the knots are represented as a
tube around a curve. It looks like a rope, but without examination
physical characteristics of the rope.

A trefoil knot, see Figs. \ref{TrefoilFirst}-\ref{TrefoilSecond}, is given by parametric equations: $x=sin(u)+2cos(2u),
y=cos(y)-2cos(2u), z=-sin(3u)$ and bending field is defined by: $p(u) = cos(3u)$ and $q(u) =
sin(3u)$.

A figure eight knot is given by the parametric equations:
$x=(2+cos(2u))*cos(3u)$, $y=(2+cos(2u))*sin(3u)$, $z=sin(4u)$. The basic and infinitesimally bent figure eight knot are given in Figs. \ref{EightFirst}-\ref{EightSecond}.

The bending fields are defined by:  $p(u) = cos(6u)$ and $q(u) = sin(6u)$.

\begin{figure}[th]
  \centerline{
  \includegraphics[width=4.6truecm]{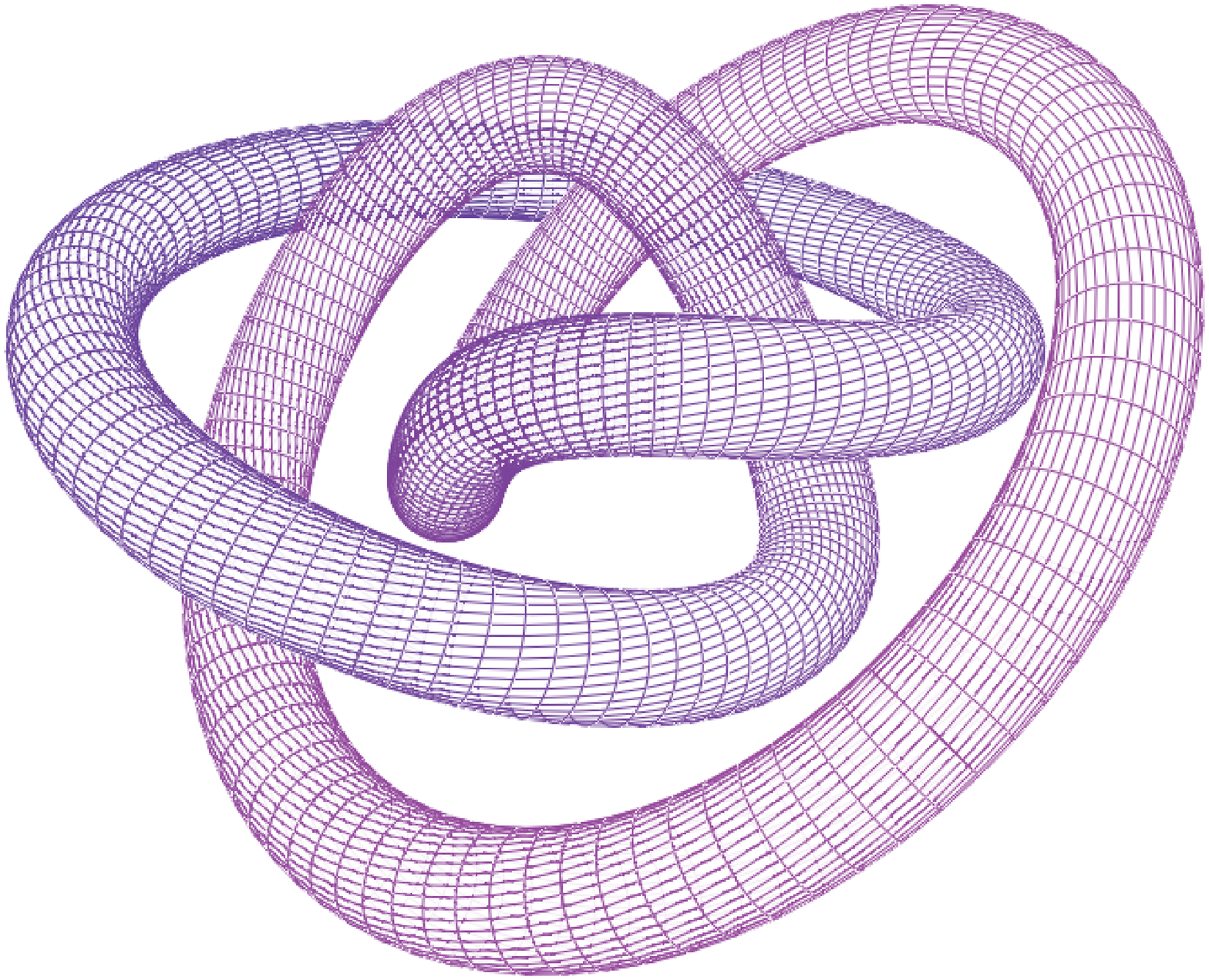}
  \includegraphics[width=4.6truecm]{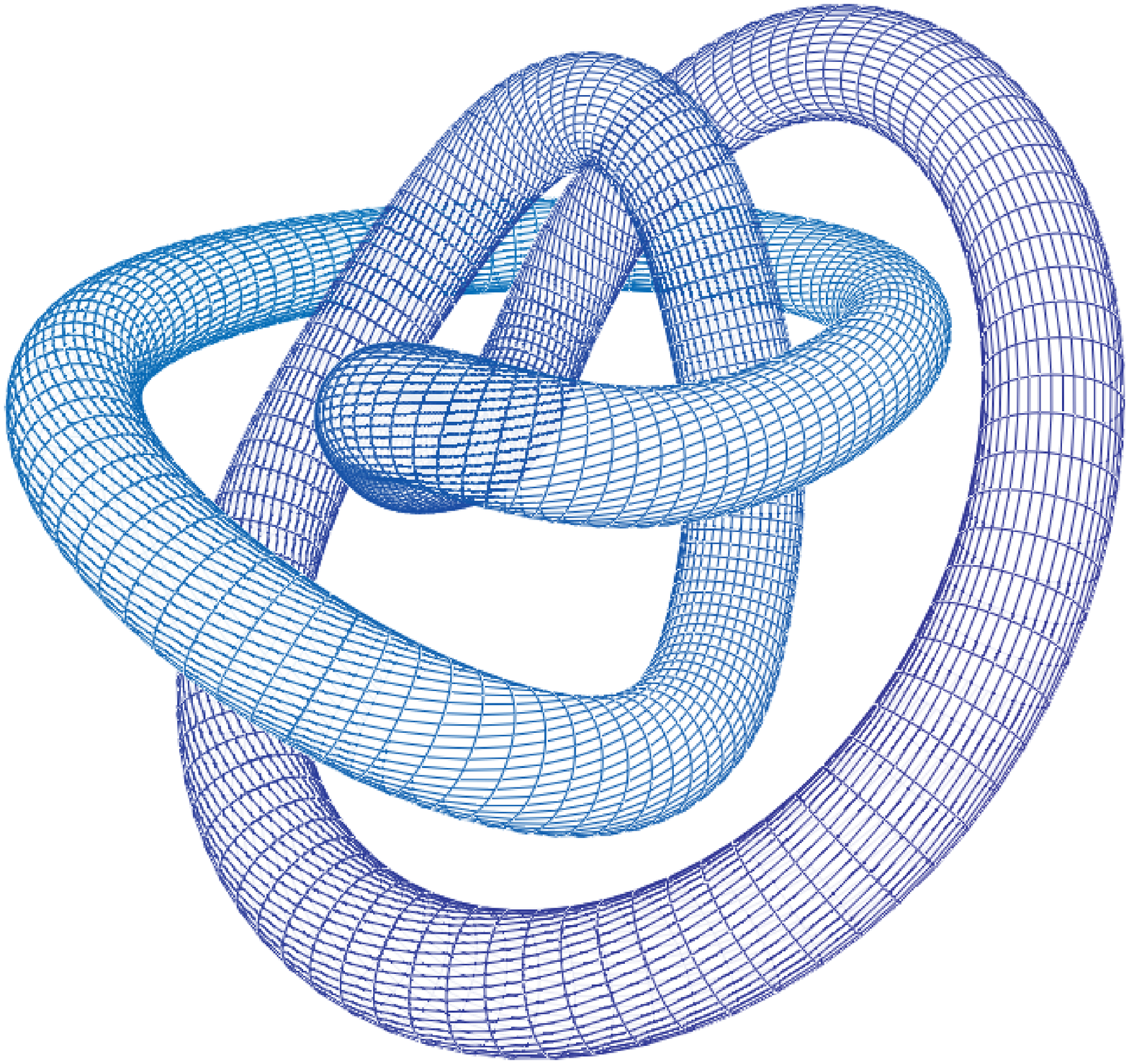}
  }
  \vspace*{8pt}
  \caption{Figure eight knot: basic and infinitesimally bent with $\epsilon = 1.4$.}\label{EightFirst}
\end{figure}
\medskip
\newpage
\begin{figure}[th]
  \centerline{
  \includegraphics[width=4.6truecm]{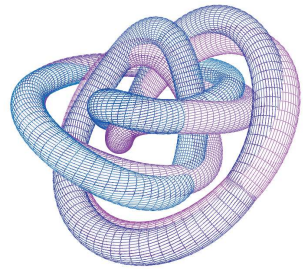}
  }
  \vspace*{8pt}
  \caption{Figure eight knot: basic and infinitesimally bent with $\epsilon = 1.4$ together.}\label{EightSecond}
\end{figure}
\smallskip

\section*{Acknowledgement}
The authors were supported from the research project 174012 of the
Serbian Ministry of Science and Quantum Topology Seminar at the
University of Illinois at Chicago. Kauffman's work was
supported by the Laboratory of Topology and Dynamics, Novosibirsk
State University (contract no. 14.Y26.31.0025 with the Ministry of
Education and Science of the Russian Federation).


\end{document}